\documentclass[11pt]{amsart}
\usepackage{amsfonts}

\newtheorem{theorem}{Theorem}
\theoremstyle{plain}
\newtheorem{acknowledgement}{Acknowledgement}

\newtheorem{corollary}{Corollary}

\newtheorem{definition}{Definition}
\newtheorem{example}{Example}

\newtheorem{remark}{Remark}

\numberwithin{equation}{section}
\input{tcilatex}

\begin{document}
\title{On the classification of $G$-graded twisted algebras}
\author{Juan D. V\'{e}lez, Luis A.Wills, Natalia Agudelo}
\maketitle

\begin{abstract}
Let $G$ denote a group and let $W$ be an algebra over a commutative ring $R$%
. We will say that $W$ is a $G$-graded twisted algebra (not necessarily
commutative, neither associative) if there exists a $G$-grading $W=\oplus
_{g\in G}W_{g}$ where each summand $W_{g}$ is a free rank one $R$ -module,
and $W$ has no monomial zero divisors (for each pair of nonzero elements%
\textit{\ }$w_{a},w_{b}$ en $W_{a}$ and $W_{b}$ their product is not zero, $%
w_{a}w_{b}\neq 0$)$.$ It is also assumed that $W$ has an identity element.

In this article, methods of group cohomology are used to study the general
problem of classification under graded isomorphisms. We give a full
description of these algebras in the associative cases, for complex and real
algebras. In the nonassociative case, an analogous result is obtained under
a symmetry condition of the corresponding associative function of the
algebra, and when the group providing the grading is finite cyclic.
\end{abstract}

\section{Introduction}

$G$-graded twisted algebras were introduced in \cite{EdLew}, and
independently in \cite{LWTPhDthesis}, as distinguished mathematical
structures which arise naturally in theoretical physics \cite{Lawt1}, \cite%
{Lawt2}, \cite{Lawt3}, \cite{Lawt4}, \cite{Lawt5}, and \cite{EdLew1}. By one
of these algebras we mean the following: Let $G$ denote a group. An $R$%
-algebra $W$ (not necessarily commutative, neither associative) will be
called a $G$\textit{-graded twisted algebra }if there exists a $G$-grading,
i.e., $W=\oplus _{g\in G}W_{g},$ with $W_{a}W_{b}\subset W_{ab}$, in which
each summand $W_{g}$ is an $R$ -module of free rank one. We assume that $W$
has an identity element $1\in W_{e}$, where $W_{e}$ denotes the graded
component corresponding to the identity element $e$ of $G.$ We also require
that $W$ has no \textit{monomial zero divisors}, i.e., for each pair of
nonzero elements\textit{\ }$w_{a}\in W_{a},$ and $w_{b}\in W_{b},$ their
product must be non zero, $w_{a}w_{b}\neq 0$ (for a general study of
nonassociative algebras, the reader may consult \cite{Schafer}.)

Besides its interest for physicists, these algebras are natural objects of
study for mathematicians, since they are related to generalizations of Lie
algebras that include transformation parameters with noncommutative and/or
nonassociative properties, as defined in \cite{SPJMS}, where generalizations
of some results of Scheunert \cite{Scheunert} on epsilon or color (super)
Lie algebras are discussed. One important case, that of $G$-graded twisted 
\textit{division} algebras (in which every nonzero element has a left and a
right inverse, but where the left and right inverses do not necessarily
coincide) already includes the classical quaternion and octonion division
algebras. They are a new remarkable class of (noncommutative and
nonassociative) division algebras over the reals. Their classification is
addressed in \cite{PapII}, \cite{PapIII}.

For the general case of $G$-graded twisted algebras over the complex
numbers, first attempts towards their classification appear in \cite{ICM2008}%
, where rudimentary techniques of group cohomology were introduced. Those
methods are exploited in this article to obtain necessary and sufficient
conditions for two algebras to be isomorphic under graded isomorphisms (see
Theorem \ref{clasifica cohomology} and its corollaries). On the other hand,
classical techniques of group representation theory readily give a complete
classification in the associative case (Corollaries \ref{july} and \ref%
{clasificacion asocitivo}.) This last result is achieved via Theorem \ref%
{fundamental}, which allows to represent these algebras as quotients of
certain group rings. A similar description is possible in the general, not
associative case, where instead of group rings the corresponding structures
are quotients of \textit{loop rings,} objects, notwithstanding, poorly
studied in the literature.

As for general nonassociative algebras, we deal in this article only with
the simplest cases. We study algebras over the complex and real numbers
which are graded by finite abelian groups, and whose associativity function
--see (\ref{rasocia}) below-- satisfies certain symmetry condition (Theorems %
\ref{complicado2} and \ref{complicado1}.)

\section{Definitions and basic concepts}

\begin{definition}
\emph{Let }$G$\emph{\ denote a group and let }$W$\emph{\ be an algebra over
a commutative ring }$R$\emph{. We will say that }$W$\emph{\ is a }$G$-graded
twisted algebra\emph{\ (not necessarily commutative neither associative) if
there exists a }$G$\emph{-grading }$W=\oplus _{g\in G}W_{g}$\emph{, with }$%
W_{a}W_{b}\subset W_{ab},$\emph{\ where each summand }$W_{g}$\emph{\ is a
free rank one }$R$\emph{\ -module. It is required that }$W$\emph{\ has no
monomial zero divisors. This condition means that for each pair of nonzero
elements\ }$w_{a}\in W_{a}$\emph{\ and }$w_{b}\in W_{b}$\emph{, }$%
w_{a}w_{b}\neq 0$\emph{. It is also assumed that }$W$\emph{\ has an identity
element }$1=w_{e}\in W_{e}$\emph{, where }$e$\emph{\ denotes the identity
element of }$G.$
\end{definition}

Since each graded component $W_{g}$ is a free rank one $R$ -module, we may
choose $w_{g}\in W_{g}$ such that $B=\{w_{g}:g\in G\}$ is a basis for $W$ as
an $R$-module. For any such fixed basis $B$ we may associate a function (the 
\emph{structure constant with respect to }$B$) $C_{B}:G\times G\rightarrow
R^{\ast }$, which takes values in $R^{\ast }=R-\{0\},$ and such that for any
pair of elements $w_{a}\in W_{a}$, $w_{b}\in W_{b}$, $%
w_{a}w_{b}=C_{B}(a,b)w_{ab}.$ We notice that $w_{a}w_{e}=C_{B}(a,e)w_{a}$,
implies $C_{B}(a,e)=1,$ for all $a$ in $G.$ The fact that $W$ has no
monomial zero divisors implies that $C_{B}$ must take values in a a
subdomain $A$ of the multiplicative group $R^{\ast }.$ If we are interested
in stressing this fact, we will write $C_{B}:G\times G\rightarrow A$. On the
other hand, if $B$ is understood, we will simply write $C:G\times
G\rightarrow A,$ omitting the subscript.

If $R$ is a field, and $A\subset R,$ a subfield, the commutative and
associative properties of $W$ can be understood by means of the following
two functions, $q:G\times G\rightarrow A,$ and $r:G\times G\times
G\rightarrow A,$ defined as:%
\begin{eqnarray}
q(a,b) &=&C(a,b)C(b,a)^{-1}  \notag \\
r(a,b,c) &=&C(b,c)C(ab,c)^{-1}C(a,bc)C(a,b)^{-1}  \label{rasocia}
\end{eqnarray}%
If $G$ is abelian, it holds that $%
w_{a}(w_{b}w_{c})=r(a,b,c)(w_{a}w_{b})w_{c} $; and that $%
w_{a}w_{b}=q(a,b)w_{b}w_{a}.$

\begin{definition}
\label{morfismo}\emph{By a }morphism\emph{\ between two }$G$\emph{-graded
twisted algebras }$W=\oplus _{g\in G}W_{g}$\emph{\ and }$V=\oplus _{g\in
G}V_{g}$\emph{\ we mean an unitarian homomorphism of }$R$\emph{-algebras }$%
\phi :W\rightarrow V$\emph{. If it also preserves the grading, i.e., }$\phi
(W_{g})\subset V_{g},$\emph{\ we say the morphism is graded.}
\end{definition}

\begin{remark}
\emph{\label{lluvia}Let us fix }$B_{1}=\{w_{g}:g\in G\}$\emph{\ and }$%
B_{2}=\{v_{g}:g\in G\}$\emph{\ bases for }$W$\emph{\ and }$V,$\emph{\
respectively. Defining a }graded morphism\emph{\ }$\phi :W\rightarrow V$%
\emph{\ amounts to giving a function }$\varphi :G\rightarrow R$\emph{\ such
that }$\phi (w_{g})=\varphi (g)v_{g}.$\emph{\ Clearly, }$\phi $\emph{\ will
be an isomorphism if and only if for all }$g,$\emph{\ }$\varphi (g)$\emph{\
is a unit in }$R.$
\end{remark}

The classification problem of $G$-graded twisted algebras admits at least
two versions. \label{isomorfismos}

\begin{enumerate}
\item $W$ and $V$ can be isomorphic as \emph{graded} algebras, that is,
there are \emph{graded} isomorphisms $\phi :W\rightarrow V$ and $\psi
:V\rightarrow W$ such that $\phi \psi $ and $\psi \phi $ is the identity.

\item $W$ and $V$ can be isomorphic as $R$-algebras, without taking into
consideration the grading.
\end{enumerate}

In this article \emph{we will restrict to the case where }$R$\emph{\ is a
field}$,$ that we will denote by $k$\emph{. It will also be assumed that }$G$%
\emph{\ is a finite group.}

\section{$G$-graded twisted algebras and Loop rings}

Let $W$ be a $G$-graded twisted algebra, and let $C:G\times G\rightarrow
A\subset k^{\ast }$ be the structure constant with respect to a fixed basis $%
B$. The function $C$ gives rise to an extension $E$ of $A$ by $G$, that will
be denoted by $A\times _{C}G$, defined as follows: $E_{W}=\{(\alpha
,g):\alpha \in A,g\in G\}\subset A\times G$. We endow this set with the
following operation: 
\begin{equation*}
(\alpha ,a).(\beta ,b)=(\alpha \beta C(a,b),ab)\,.
\end{equation*}%
It is easy to see that whenever $W$ is associative, $E_{W}$ is a group,
actually an extension of $A$ by $G.$ In general, if $W$ is not associative,
then $E_{W}$ turns out to be a \textit{loop}, a structure which is almost a
group, except that its binary operation is not necessarily associative, but
where there is still an identity element, and where each element has a right
and a left inverse, not necessarily equal. As in the case of an ordinary
group, \textit{the loop ring} $R=k[E_{W}]$ may be defined analogously.

The structure of any associative $G$-graded twisted algebra can be
understood using the following fundamental result.

\begin{theorem}
\label{fundamental}Let $W=\oplus _{g\in G}W_{g}$, and let $C:G\times
G\rightarrow A$ be the structure constant related to a fixed choice of basis
for $W.$ Let $A\times _{C}G$ denote the extension of $A$ by $G$ defined
above, and let $R=k\left[ A\times _{C}G\right] $ be the Loop ring over $%
A\times _{C}G$. Let us define the vector subspace $I$ of $R$ generated by
all elements of the form $I=\left\langle (\alpha ,g)-\alpha (1,g):\alpha \in
A,g\in G\right\rangle $. Then $I\subset R$ is a bilateral ideal and $R/I$ is
a $k$-algebra isomorphic to $W.$
\end{theorem}

\begin{proof}
It is an elementary fact that $I$ is a bilateral ideal of $R.$

Now, for $(\alpha ,a)\in A\times _{C}G$ let us define $\varphi :A\times
_{C}G\rightarrow W^{\ast }$ as the function which sends $(\alpha ,a)$ to $%
\alpha v_{a},$where $W^{\ast }$ denotes the invertible elements of $W.$
Clearly, $\varphi ((\alpha ,a)(\beta ,b))=\varphi (\alpha ,a)\varphi (\beta
,b),$ and consequently $\varphi $ can be extended to a $k$-linear map (that
we will denote by the same letter) $\varphi :R\rightarrow W.$ An easy
computation shows that $I\subset Ker(\varphi )$, and consequently $\varphi $
descends to the quotient $\varphi :R/I\rightarrow W,$ sending the class of $%
\tsum \lambda _{(\alpha ,a)}(\alpha ,a)$ into $\tsum \lambda _{(\alpha
,a)}\alpha v_{a}$. An inverse map can be defined explicitly by extending
linearly the map $\phi :W\rightarrow R/I$ that sends each element $v_{a}$ to
the class of $(1,a)$. $\phi $ is indeed a homomorphism of $k$ algebras: ("$-$%
" denotes the class of an element) 
\begin{eqnarray}
\phi (\tsum\limits_{a\in G}\lambda _{a}v_{a}\tsum\limits_{b\in G}\lambda
_{b}v_{b}) &=&\phi (\tsum\limits_{a,b\in G}\lambda _{a}\lambda
_{b}C(a,b)v_{ab})  \notag \\
&=&\phi (\tsum\limits_{g\in G}(\tsum\limits_{a+b=g}C(a,b)\lambda _{a}\lambda
_{b})v_{ab})  \notag \\
&=&\tsum\limits_{g\in G}\left( \tsum\limits_{a+b=g}C(a,b)\lambda _{a}\lambda
_{b}\right) (1,g)^{-}  \label{marca}
\end{eqnarray}%
A similar computation shows that (\ref{marca}) is equal to $\phi
(\tsum\limits_{a\in G}\lambda _{a}v_{a})\phi (\tsum\limits_{b\in G}\lambda
_{b}v_{b}).$
\end{proof}

\section{Associative $G$-graded algebras over $\mathbb{C}$}

Standard techniques of group representation theory allows us to completely
classify all $G$-graded twisted associative $\mathbb{C}$-algebras. For this,
let us fix a basis $B$ and let $C:G\times G\rightarrow A\subset 
\mathbb{C}
^{\ast }$ be the structure constant of $W$ with respect to $B.$ As observed
above, whenever $W$ is associative $E_{W}=A\times _{C}G$ is not only a loop,
but a group. Let $R=%
\mathbb{C}
\left[ A\times _{C}G\right] $ be its associated group ring. It is a standard
result that $R$ is the regular representation of $E_{W},$ and that if $R$
decomposes into irreducible representations $R=V_{1}^{a_{1}}\oplus \cdots
\oplus V_{r}^{a_{r}},$ with $V_{i}\neq V_{j}$, then the exponents $a_{i}$
can be computed in terms of the characters $\chi _{R}$ and $\chi _{V_{i}}$
as 
\begin{equation*}
a_{i}=\left\langle \chi _{R},\chi _{V_{i}}\right\rangle =\frac{1}{\left\vert
A\times _{C}G\right\vert }\tsum_{g\in G}\overline{\chi _{R}(g)}\chi
_{V_{i}}(g)=\dim V_{i}.
\end{equation*}%
(See \cite{Gordon Martin}, \cite{Sr}). Hence, the homomorphism $\varphi :%
\mathbb{C}
\left[ A\times _{C}G\right] \rightarrow \oplus _{i=1}^{r}Hom_{%
\mathbb{C}
}(V_{i},V_{i})$ that send each element $s\in 
\mathbb{C}
\left[ A\times _{C}G\right] $ into $(\psi _{i})_{i=1,\ldots .r}$, where $%
\psi _{i}:V_{i}\rightarrow V_{i}$ denotes multiplication by $s,$ is in fact
an isomorphism of $\mathbb{C}$-algebras (no grading involved) \cite{Sr}.
With notation as in Theorem \ref{fundamental}:

\begin{theorem}
There is an isomorphism (not necessarily graded) of $%
\mathbb{C}
$-algebras 
\begin{eqnarray*}
\varphi \overset{}{:}%
\mathbb{C}
\left[ A\times _{C}G\right] /I &\rightarrow &\oplus _{i=1}^{r}Hom_{%
\mathbb{C}
}(V_{i},V_{i})/J_{i} \\
\left[ s\right] &\mapsto &(\psi _{i})_{i=1,\ldots .r}
\end{eqnarray*}%
where $J=\varphi (I)$ is isomorphic to a product $J_{1}\times \cdots \times
J_{n}$, of bilateral ideals $J_{i}\subset Hom_{%
\mathbb{C}
}(V_{i},V_{i}).$
\end{theorem}

\begin{proof}
The ideal $J$ decomposes as a product of bilateral ideals $J=J_{1}\times
\cdots \times J_{r},$ with $J_{i}\subset Hom_{%
\mathbb{C}
}(V_{i},V_{i})$. Hence, 
\begin{equation*}
W\cong R/I=\oplus _{i=1}^{r}Hom_{%
\mathbb{C}
}(V_{i},V_{i})/J_{i}.
\end{equation*}%
But each one of the algebras $Hom_{%
\mathbb{C}
}(V_{i},V_{i})$ is simple and consequently $J_{i}=0$ or $J_{i}=Hom_{%
\mathbb{C}
}(V_{i},V_{i})$. Thus, $W\cong \oplus _{i}Hom_{%
\mathbb{C}
}(V_{i},V_{i})$ where the sum occurs only for those $i$ such that $%
J_{i}=(0). $
\end{proof}

\begin{corollary}
\label{july}Let $W$ be a $G$-graded twisted associative $\mathbb{C}$%
-algebra. Then $W$ is isomorphic as a $%
\mathbb{C}
$-algebra to a finite product of matrix algebras \thinspace $%
Mat_{n_{i}\times n_{i}}(%
\mathbb{C}
),$ where $n_{i}=\dim V_{i}.$The algebra $W$ is commutative if and only if $%
n_{i}=\dim V_{i}=1,$ and therefore, if and only if $W\simeq \mathbb{C}\times
\cdots \times \mathbb{C}.$
\end{corollary}

\subsection{Associative $G$-graded algebras over $\mathbb{R}$}

For the real case we proceed in a similar manner as in the last section. For
any group $G,$ the group ring $\mathbb{R}[G]$ is isomorphic to a finite
direct sum of $\mathbb{R}$-algebras of the form $Hom_{D_{i}}(V_{i},V_{i})$
where $D_{i}$ is a division algebra over the reals. \label{remarquito}
Moreover, $D_{i}=$ $Hom_{%
\mathbb{R}
}(V_{i},V_{i})^{G}$ (see \cite{Sr}). This immediately yields the following:

\begin{theorem}
\label{clasificacion asocitivo}Let $W=\oplus _{g\in G}W_{g}$ be a $G$%
-graded, twisted, associative $\mathbb{R}$-algebra. Then $W$ is isomorphic
to a direct sum $\oplus _{i=1}^{r}Hom_{D_{i}}(V_{i},V_{i}),$ where $D_{i}$
denotes one of the division rings $%
\mathbb{R}
,\mathbb{C},$ or $\mathbb{H}$, the quaternions.
\end{theorem}

\begin{proof}
We already know that $W\simeq \mathbb{R}[A\times _{C}G]/I$, for some
bilateral ideal $I.$ From the remark above (\ref{remarquito}) it follows
that 
\begin{equation*}
\mathbb{R}[A\times _{C}G]/I\simeq \oplus
_{i=1}^{n}Hom_{D_{i}}(V_{i},V_{i})/J_{i},
\end{equation*}%
for bilateral ideals $J_{i}$ of $Hom_{D_{i}}(V_{i},V_{i}),$ where $D_{i}=$ $%
Hom_{%
\mathbb{R}
}(V_{i},V_{i})^{A\times _{C}G}$ is a division (associative) algebra over the
reals. But it is a well known fact that $D_{i}$ must equal to one of the
algebras $%
\mathbb{R}
,\mathbb{C}$ or $\mathbb{H}$. Thus, the result follows.
\end{proof}

\section{Graded Morphisms and group Cohomology}

In this section we study the problem of determining when two $G$-graded
twisted $k$-algebras are isomorphic under a graded isomorphism. The
following theorem provides necessary and sufficient conditions for two
algebras to be graded-isomorphic in terms of the second group cohomology $%
H^{2}(G,k^{\ast }).$ For the basic notions about group cohomology, the
reader may consult \cite{AtiyahWall}, \cite{Baba}.

\begin{theorem}
\label{clasifica cohomology}Let $V=\oplus _{g\in G}V_{g}$ and $W=\oplus
_{g\in G}W_{g}$ two $G$-graded $k$-algebras. Let us fix bases $B_{1}$ and $%
B_{2}$ for $V$ and $W$, respectively, and let $C_{1}$ and $C_{2}$ be the
associated structure constants. Then $V$ is isomorphic to $W$ if and only if
the function $C_{1}C_{2}^{-1}$ is in the kernel of $d^{2}:C^{2}(G,k^{\ast
})\rightarrow C^{3}(G,k^{\ast })$ and the class $[C_{1}C_{2}^{-1}]$ is
trivial in $H^{2}(G,k^{\ast })$.
\end{theorem}

\begin{proof}
Let us suppose that $V$ and $W$ are isomorphic as $k$-algebras under a
grading preserving isomorphism $\phi :V\rightarrow W$. This implies that
there exists $\varphi :G\rightarrow k^{\ast }$ which sends each vector $%
v_{g} $ into $\varphi (g)w_{g}$. But $\phi $ a homomorphism implies that $%
\phi (C_{1}(a,b)v_{a+b})=\varphi (a)w_{a}\varphi (b)w_{b},$ and
consequently, $C_{1}(a,b)\varphi (ab)w_{a+b}=\varphi (a)\varphi
(b)C_{2}(a,b)w_{a+b}.$ From this we obtain 
\begin{equation}
C_{1}(a,b)C_{2}^{-1}(a,b)=\varphi (a)\varphi (ab)^{-1}\varphi (b).
\label{delta}
\end{equation}%
Notice that $d^{1}\varphi (a,b)=\varphi (b)\varphi ^{-1}(ab)\varphi (a),$
and therefore $C_{1}C_{2}^{-1}$ belongs to the image of $d^{1}:C^{1}(G,k^{%
\ast })\rightarrow C^{2}(G,k^{\ast }).$ Thus, $d^{2}(C_{1}C_{2}^{-1})=1,$
and $[C_{1}C_{2}^{-1}]=1$ in $H^{2}(G,k^{\ast }).$

Reciprocally, if $d^{2}(C_{1}C_{2}^{-1})=1,$ and $[C_{1}C_{2}^{-1}]=1$ in $%
H^{2}(G,k^{\ast }),$ then there exists $\varphi :G\rightarrow k^{\ast }$
such that $d^{1}\varphi =C_{1}C_{2}^{-1}$ and consequently equation (\ref%
{delta}) holds. It then follows that the function $\phi :V\rightarrow W$
defined on the basis $B_{1}$ as $\phi (v_{g})=\varphi (g)w_{g}$ is a
homomorphism of $k$-algebras, which is injective, since so it is $\phi $.
But $V$ are $W$ are $k$-vector spaces of the same dimension (equal to $%
\left\vert G\right\vert )$. Hence, $\phi $ is an isomorphism.
\end{proof}

\begin{remark}
\emph{If }$V$\emph{\ and }$W$\emph{\ are associative, then }$%
d^{2}C_{1}=d^{2}C_{2}=1,$\emph{\ since both these terms are equal to the
associativity function in (\ref{rasocia}). In this case, the class of each }$%
C_{i}$\emph{\ is an element of }$H^{2}(G,k^{\ast })$\emph{\ and the
condition }$[C_{1}C_{2}^{-1}]=1$\emph{\ is equivalent to }$[C_{1}]=[C_{2}]$%
\emph{\ in }$H^{2}(G,k^{\ast })$\emph{.}
\end{remark}

\begin{corollary}
\label{idependiente}Let $W=\oplus _{g\in G}W_{g}$ be a $G$-graded $k$%
-algebra, and let be $B$ and $B^{\prime }$ be bases for $W$, with associated
constant structures $C$ and $C^{\prime }.$ Let $r$ and $r^{\prime }$ be the
corresponding functions as defined above. Then $r=r^{\prime }.$ In other
words, the associativity function of $W$ does not depend on any chosen basis.
\end{corollary}

\begin{proof}
The identity isomorphism $I:W\rightarrow W$ is trivially graded. By the
previous theorem, $[C^{\prime }C^{-1}]=1$ and consequently $C^{\prime
}C^{-1}\in im(d^{1}).$ Hence, $d^{2}(C^{\prime }C^{-1})=1$, and therefore $%
r^{\prime }=d^{2}(C^{\prime })=d^{2}(C)=r.$
\end{proof}

Let us notice now that if $C_{1}$ and $C_{2}$ take values in a subdomain $%
A\subset k^{\ast }$, then $C_{1}C_{2}^{-1}\in C^{2}(G,A).$ Hence, if $%
d^{2}(C_{1}C_{2}^{-1})=1$, it makes sense to talk about the class $%
[C_{1}C_{2}^{-1}]\in H^{2}(G,A).$ The following theorem gives a criterion in
terms of $H^{2}(G,A)$ to determine when $V$ and $W$ are isomorphic. This may
be useful in many cases where $A$ is a finite subgroup of $k^{\ast }.$

\begin{theorem}
\label{dos}$\phi :V\rightarrow W$ is a (graded) isomorphism if and only if $%
d^{2}(C_{1}C_{2}^{-1})=1,$ and $[C_{1}C_{2}^{-1}]\in \ker (i_{2})$, where $%
i_{2}:H^{2}(G,A)\rightarrow H^{2}(G,k^{\ast })$ denotes the homomorphism in
cohomology induced by the inclusion $i:A\rightarrow k^{\ast }$.
\end{theorem}

\begin{proof}
The short exact sequence of groups%
\begin{equation*}
1\rightarrow A\overset{i}{\rightarrow }k\overset{\pi }{\rightarrow }k^{\ast
}/A\rightarrow 1.
\end{equation*}%
induces an exact sequence of complexes 
\begin{equation}
1\rightarrow C^{\bullet }(G,A)\overset{i_{\bullet }}{\rightarrow }%
C^{^{\bullet }}(G,k^{\ast })\overset{\pi _{\bullet }}{\rightarrow }%
C^{^{\bullet }}(G,k^{\ast }/A)\rightarrow 1,  \tag{(*)}
\end{equation}%
where we may identify the quotient $C^{^{\bullet }}(G,k^{\ast })/C^{\bullet
}(G,A)$ with $C^{^{\bullet }}(G,k^{\ast }/A)$ via the isomorphism that sends
the class of $h:G\rightarrow k^{\ast }$ into $\pi \circ h.$ By Theorem \ref%
{clasifica cohomology}, $V$ and $W$ are graded isomorphic, if and only if $%
d^{2}(C_{1}C_{2}^{-1})=1,$ and $[C_{1}C_{2}^{-1}]=1$ in $H^{2}(G,k^{\ast }).$
But looking at the long exact sequence for cohomology 
\begin{equation*}
\cdots \rightarrow H^{1}(G,k^{\ast })\overset{\pi _{1}}{\rightarrow }%
H^{1}(G,k^{\ast }/A)\overset{\delta }{\rightarrow }H^{2}(G,A)\overset{i_{2}}{%
\rightarrow }H^{2}(G,k^{\ast })\rightarrow \cdots
\end{equation*}%
we see that this occurs precisely when $i_{2}([C_{1}C_{2}^{-1}])=1.$
\end{proof}

\begin{example}
\emph{If }$G$\emph{\ denotes a cyclic group of order }$n,$\emph{\ then it is
well know that }$H^{2}(G,C^{\ast })=%
\mathbb{C}
^{\ast }/(%
\mathbb{C}
^{\ast })^{n}=\{1\}$\emph{. Hence, if }$V$\emph{\ and }$W$\emph{\ are }$G$%
\emph{-graded associative algebras with structure constants given by }$%
C_{1},C_{2}:G\times G\rightarrow A\subset C^{\ast }$\emph{, then }$%
[C_{1}][C_{2}]^{-1}=1$\emph{\ and consequently they are isomorphic. It
readily follows that }$C[t]/(t^{n}-1)=\oplus _{r=0}^{n-1}Ct^{r}$\emph{\ is a
representative of the unique isomorphism class.}
\end{example}

\begin{remark}
\emph{If }$k=R$\emph{, then }$H^{2}(G,R^{\ast })=\{1\},$\emph{\ if }$n$\emph{%
\ is odd, and it is equal to }$\{1,-1\},$\emph{\ if }$n$\emph{\ is even. In
the first case, there exists a unique real associative algebra }$%
R[t]/(t^{n}-1)$\emph{. In the second case, there are exactly two algebras,
given by }$R[t]/(t^{n}-1),$\emph{\ and }$R[t]/(t^{n}+1).$\emph{\ On the
other hand, if }$V$\emph{\ and }$W$\emph{\ have structure constants }$%
C_{1},C_{2}:G\times G\rightarrow A$\emph{, where }$\left\vert A\right\vert $%
\emph{\ and }$\left\vert G\right\vert $\emph{\ are relatively prime
integers, then }$H^{2}(G,A)=\{1\}$\emph{. From the previous theorem, if }$%
d^{2}(C_{1}C_{2}^{-1})=1$\emph{\ then }$V$\emph{\ and }$W$\emph{\ are
isomorphic as graded algebras. In the general case }$%
d^{2}(C_{1}C_{2}^{-1})=d^{2}(C_{1})d^{2}(C_{2})^{-1}=r_{1}r_{2}^{-1}$\emph{,
where }$r_{i}$\emph{\ is the associativity function of }$C_{i}.$\emph{\
Thus, }$r_{1}=r_{2}$\emph{\ if and only if }$V$\emph{\ and }$W$\emph{\ are
isomorphic.}
\end{remark}

This proves the following.

\begin{theorem}
Let $W^{1}=\oplus _{g\in G}W_{g}^{1},$ and $W^{2}=\oplus _{g\in G}W_{g}^{2},$
be $G$-graded $k$-algebras over a finite group $G$. Let $r_{1,}r_{2}:G^{3}%
\rightarrow A$ be the corresponding associativity functions. If $\left\vert
A\right\vert $ and $\left\vert G\right\vert $ are relatively prime integers,
then $W^{1}$ and $W^{2}$ are isomorphic as graded algebras if and only if $%
r_{1}=r_{2}.$ In particular, if $W^{1}$ and $W^{2}$ are associative, and if $%
\left\vert A\right\vert $ and $\left\vert G\right\vert $ are relatively
prime integers, then $W^{1}$ and $W^{2}$ are isomorphic as graded algebras.
\end{theorem}

\section{Some classification results in the nonassociative case}

In this section we shall give a complete classification under graded
isomorphisms of all $G$-graded twisted algebras, when $G$ is a finite cyclic
group, and under the condition that their associative function satisfies the
following symmetric condition: $r(a,b,c)=r(b,a,c),$ for all $a,b,c\in G$ %
\label{simetria}.

We start with a general discussion. Let $G$ be any abelian group $G$, and
let $W=\oplus _{g\in G}W_{g}$ be a $G$-graded twisted $k$-algebra. Let us
fix a basis $B=\{v_{g}:g\in G\}$, and denote by $T_{g}:W\rightarrow W$ the $%
k $-linear map defined by multiplying (on the left) by $v_{g}$. If $a$ and $b
$ are arbitrary elements of $G$, then for any $v_{g}\in B$ 
\begin{eqnarray*}
T_{a}(T_{b}(v_{g})) &=&T_{a}(v_{b}v_{g}) \\
&=&v_{a}(v_{b}v_{g}) \\
&=&r(a,b,g)(v_{a}v_{b})v_{g} \\
&=&r(a,b,g)C(a,b)v_{ab}v_{g} \\
&=&r(a,b,g)C(a,b)T_{ab}(v_{g}),
\end{eqnarray*}%
and consequently, 
\begin{equation}
T_{ab}(v_{g})=r(a,b,g)^{-1}C(a,b)^{-1}T_{a}(T_{b}(v_{g})).  \label{prim}
\end{equation}%
Similarly, 
\begin{equation}
T_{ba}(v_{g})=r(b,a,g)^{-1}C(b,a)^{-1}T_{b}(T_{a}(v_{g})).  \label{sec}
\end{equation}%
Since $G$ is abelian, $T_{ab}=T_{ba}$. Therefore, 
\begin{equation*}
T_{a}(T_{b}(v_{g}))=r(a,b,g)C(a,b)r(b,a,g)^{-1}C(b,a)^{-1}T_{b}(T_{a}(v_{g})).
\end{equation*}%
Using the symmetry condition (\ref{simetria}) $r(a,b,g)=r(b,a,g)$, it
follows that 
\begin{equation*}
T_{a}(T_{b}(v_{g}))=q(a,b)T_{b}(T_{a}(v_{g})).
\end{equation*}%
Hence, for any $x\in W$ (not necessarily homogeneous)%
\begin{equation}
T_{a}(T_{b}(x))=q(a,b)T_{b}(T_{a}(x))  \label{maravilla}
\end{equation}%
In a similar way,%
\begin{equation}
T_{b}(T_{a}(x))=q(b,a)T_{a}(T_{b}(x))\text{.}  \label{maravilla1}
\end{equation}%
Let $G$ be a cyclic group of order $n$. Fix any generator $a$ of $G.$ Let $%
w_{a}$ be any non zero element in $W_{a}.$ Define inductively $%
w_{a^{k}}=w_{a}\cdot w_{a^{k-1}},$ where $w_{a^{0}}=\alpha $, for some
arbitrary fixed $\alpha \neq 0$ in $\mathbb{C}.$ Define $v_{a^{k}}=\beta
^{k}w_{a^{k}}$, where $\beta $ denotes any primitive $n$-th root of unity of 
$\alpha ^{-1}$, ($\beta ^{n}=1/\alpha $)$.$ If $k=n,$ we see that $%
v_{a^{0}}=\beta ^{n}w_{a^{0}}=\beta ^{n}\alpha =1,$ and

\begin{equation*}
v_{a}v_{a^{k-1}}=\beta ^{k}w_{a}w_{a^{k-1}}=\beta ^{k}w_{a^{k}}=v_{a^{k}}.
\end{equation*}%
The basis $B=\{1,v_{a},\ldots ,v_{a^{n-1}}\}$ will be called the \textit{%
standard basis} for $W.$ In what follows, the structure constant $%
C_{B}(a^{r},a^{k})$\emph{\ }will be denoted by\emph{\ }$C(a^{r},a^{k}).$
Notice that $C(a,a^{r})=1$, for all $r=0,\ldots n-1$.

Since $T_{a}$ sends $v_{a^{k}}$ into $v_{a^{k+1}},$ the linear map $T_{a}$
must permute the basis $B$ cyclically, and consequently the minimal
polynomial for $T_{a}$ must be $Y^{n}-1=0$. Hence, the eigenvalues of $T_{a}$
are precisely the $n$-th complex roots of unity $\{\omega _{1},\ldots
,\omega _{n}\}$. For each $1\leq $ $j\leq n,$ let $z_{j}=\sum_{k=0}^{n-1}%
\omega _{j}^{k}v_{a^{k}}.$

If $t\equiv s$ mod $n$, then $\omega _{j}^{t}v_{a^{t}}=\omega
_{j}^{s}{}v_{a^{s}}$, since $t=s+nq$ implies 
\begin{equation*}
\omega _{j}^{t}v_{a^{t}}=\omega _{j}^{nq}\omega _{j}^{s}v_{a^{s}}=1\cdot
\omega _{j}^{s}v_{a^{s}}.
\end{equation*}%
Consequently, the sum $\sum_{k=0}^{n-1}\omega _{j}^{k}v_{a^{k}}$ can be also
be written as $\sum_{k\in \mathbb{Z}_{n}}\omega _{j}^{k}v_{a^{k}}.$ Now, $%
\omega _{j}{}^{-1}$ is an eigenvalue associated to $z_{j}$ because 
\begin{eqnarray*}
T_{a}(z_{j}) &=&\sum_{k\in \mathbb{Z}_{n}}\omega
_{j}^{k}v_{a}v_{a^{k}}=\omega _{j}{}^{-1}\sum_{k+1\in \mathbb{Z}_{n}}\omega
_{j}{}^{k+1}v_{a^{k+1}} \\
&=&\omega _{j}{}^{-1}\sum_{s\in \mathbb{Z}_{n}}^{n}\omega
_{j}^{s}v_{a^{s}}=\omega _{j}{}^{-1}z_{j}.
\end{eqnarray*}%
On the other hand, by (\ref{maravilla1})

\begin{eqnarray*}
T_{b}(T_{a}(z_{j})) &=&q(b,a)T_{a}(T_{b}(z_{j})) \\
T_{b}(\omega _{j}{}^{-1}z_{j}) &=&q(b,a)T_{a}(T_{b}(z_{j})) \\
\omega _{j}{}^{-1}T_{b}(z_{j}) &=&q(b,a)T_{a}(T_{b}(z_{j})) \\
q(a,b)\omega _{j}{}^{-1}T_{b}(z_{j}) &=&T_{a}(T_{b}(z_{j})).
\end{eqnarray*}%
Hence, $T_{b}(z_{j})$ is a eigenvector of the eigenvalue $q(a,b)\omega
_{j}{}^{-1}$. Since $T_{a}$ has $n$ different eigenvalues, there must exist $%
\omega _{i}$ and $z_{i}$ such that $\omega _{i}{}^{-1}=q(a,b)\omega
_{j}{}^{-1}$ and $T_{b}(z_{j})=\eta _{b,j}z_{i}$, for some $\eta _{b,j}\neq
0 $ en $\mathbb{C}.$

If we take $b=a^{r}$ we obtain%
\begin{eqnarray}
T_{a^{r}}(z_{j}) &=&T_{a^{r}}(\tsum_{k\in \mathbb{Z}_{n}}\omega
_{j}^{k}v_{a^{k}})=\tsum_{k\in \mathbb{Z}_{n}}\omega
_{j}^{k}C(a^{r},a^{k})v_{a^{r+k}}  \label{L1} \\
&=&\omega _{j}^{-r}\tsum_{k+r\in \mathbb{Z}_{n}}^{n-1}\omega
_{j}^{k+r}C(a^{r},a^{k})v_{a^{r+k}}  \label{LL1} \\
&=&\omega _{j}^{-r}\tsum_{s\in \mathbb{Z}_{n}}\omega
_{j}^{s}C(a^{r},a^{s-r})v_{a^{s}},  \label{LL4}
\end{eqnarray}%
($s=k+r$). On the other hand,%
\begin{eqnarray}
T_{a^{r}}(z_{j}) &=&\eta _{a^{r},j}z_{i}=\eta _{a^{r},j}\sum_{s\in \mathbb{Z}%
_{n}}\omega _{i}^{s}v_{a^{s}}  \label{LL33} \\
&=&\sum_{s\in \mathbb{Z}_{n}}\eta _{a^{r},j}q(a^{r},a)^{s}\omega
_{j}^{s}v_{a^{s}}.
\end{eqnarray}%
Comparing coefficients in (\ref{LL4}) and (\ref{LL33}) we see that 
\begin{equation}
\omega _{j}^{s-r}C(a^{r},a^{s-r})=\eta _{a^{r},j}q(a^{r},a)^{s}\omega
_{j}^{s}  \notag
\end{equation}%
and consequently%
\begin{equation}
C(a^{r},a^{s-r})=\eta _{a^{r},j}q(a^{r},a)^{s}\omega _{j}{}^{r}.
\label{primera1}
\end{equation}

In particular, if $s=r,$ we see that $1=C(a^{r},a^{0})=\eta
_{a^{r},j}q(a^{r},a)^{r}\omega _{j}{}^{r}.$ Thus, $\eta
_{a^{r},j}=q(a,a^{r})^{r}\omega _{j}{}^{-r}.$ Substituting in (\ref{primera1}%
) we get $C(a^{r},a^{s-r})=q(a,a^{r})^{r}\omega
_{j}^{-r}q(a^{r},a)^{s}\omega _{j}^{r}$, and consequently $%
C(a^{r},a^{s-r})=q(a^{r},a)^{s-r}.$

If we let $k=s-r$ we see that%
\begin{equation*}
C(a^{r},a^{k})=q(a^{r},a)^{k}=C(a^{r},a)^{k}C(a,a^{r})^{-k}=C(a^{r},a)^{k},
\end{equation*}%
(using that $C(a,a^{r})=1$). Thus, 
\begin{eqnarray}
r(a^{r},a^{s},a^{t})
&=&C(a^{s},a)^{t}C(a^{r+s},a)^{-t}C(a^{r},a)^{s+t}C(a^{r},a)^{-s}
\label{super} \\
&=&(C(a^{s},a)C(a^{r},a)C(a^{r+s},a)^{-1})^{t}  \label{super2}
\end{eqnarray}%
Let $G$ be a cyclic group of order $n$ and let $W=\oplus _{g\in G}W_{g},$ a $%
G$-graded twisted over $\mathbb{C}.$ Let us assume that $r$ is symmetric in
the first two entries. If $B$ denotes the standard basis for $W,$ then $W$
is completely determined by the function $f:G\rightarrow A,$ given by $%
f(a^{r})=C(a^{r},a),$ where $a$ is any fixed generator of $G.$ Hence, there
are precisely $\left\vert A\right\vert ^{n-2}$ non isomorphic $G$-graded
twisted algebras.

\begin{example}
\emph{Let }$Z_{4}=\left\{ a,a^{2},a^{3},a^{4}=1\right\} $\emph{\ and }$%
W=\oplus _{r=0}^{3}W_{a^{r}},$\emph{\ a }$C$\emph{-algebra. Let }$A=\left\{
1,-1\right\} .$\emph{\ The minimal polynomial for }$T_{a}$\emph{\ is }$%
Y^{4}-1=0,$\emph{\ with eigenvalues }$\omega _{1}=-i,\omega _{2}=-1,\omega
_{3}=i,\omega _{4}=1$\emph{. It is easy to see that the corresponding
eigenvectors are }%
\begin{eqnarray*}
z_{1} &=&1+iv-v_{2}-iv^{3} \\
z_{2} &=&1-v+v^{2}-v^{3} \\
z_{3} &=&1-iv-v^{2}+iv^{3} \\
z_{4} &=&1+v+v^{2}+v^{3}.
\end{eqnarray*}%
\emph{Hence,}%
\begin{eqnarray*}
T_{a}(z_{1}) &=&-iz_{1} \\
T_{a}(z_{2}) &=&-z_{2} \\
T_{a}(z_{3}) &=&iz_{3} \\
T_{a}(z_{4}) &=&z_{4}
\end{eqnarray*}%
\emph{Case }$1:$\emph{\ Suppose }$C(a^{2},a)=-1$\emph{\ and }$C(a^{3},a)=1$%
\emph{. Therefore }%
\begin{eqnarray*}
T_{a}(T_{a^{2}}(z_{1})) &=&-T_{a^{2}}(T_{a}(z_{1})) \\
T_{a}(T_{a^{2}}(z_{1})) &=&iT_{a^{2}}(z_{1})
\end{eqnarray*}%
\emph{Then, }$T_{a^{2}}(z_{1})$\emph{\ is an eigenvector associated to the
eigenvalue }$i$\emph{; that is, }$T_{a^{2}}(z_{1})=\beta (a^{2},1)z_{3}.$%
\emph{\ Hence, }%
\begin{equation*}
\beta (a^{2},1)=q(a,a^{2})^{2}\omega _{1}^{-2}=(-1)^{2}(i)^{-2}=-1
\end{equation*}%
\emph{and the structure constants are given by }%
\begin{eqnarray*}
C(a^{2},a^{2}) &=&q(a^{2},a)^{2}=1 \\
C(a^{2},a^{3}) &=&q(a^{2},a)^{3}=-1 \\
C(a^{2},a) &=&q(a^{2},a)=-1
\end{eqnarray*}%
\emph{Similarly, }$T_{a^{2}}(z_{2})$\emph{\ is an eigenvector of the
eigenvalue }$1,$\emph{\ i.e., }$T_{a^{2}}(z_{2})=\beta (a^{2},2)z_{4},$\emph{%
\ with }$\beta (a^{2},2)=1$\emph{. And }$T_{a^{2}}(z_{3})$\emph{\ is an
eigenvector of the eigenvalue }$-i,$\emph{\ i.e., }$T_{a^{2}}(z_{3})=\beta
(a^{2},3)z_{1},$\emph{with }$\beta (a^{2},3)=-1$\emph{. Finally, }$%
T_{a^{2}}(z_{4})$\emph{\ is an eigenvector of the eigenvalue }$-1:$\emph{\ }$%
T_{a^{2}}(z_{4})=\beta (a^{2},4)z_{2},$\emph{\ with }$\beta (a^{2},4)=1$%
\emph{. Now let }$b=a^{3}$\emph{; since }$q(a,a^{3})=1,$\emph{\ then }%
\begin{eqnarray*}
T_{a}(T_{a^{3}}(z_{1})) &=&T_{a^{3}}(T_{a}(z_{1})) \\
T_{a}(T_{a^{3}}(z_{1})) &=&-iT_{a^{3}}(z_{1})
\end{eqnarray*}%
\emph{Then }$T_{a^{3}}(z_{1})$\emph{\ is an eigenvector associated to }$-i,$%
\emph{\ i.e., }$T_{a^{3}}(z_{1})=\beta (a^{3},1)z_{1}.$ \emph{Thus, }%
\begin{equation*}
\beta (a^{3},1)=q(a,a^{3})^{3}\omega _{1}^{-3}=(1)^{3}(i)^{-3}=i,
\end{equation*}%
\emph{and the structure constants are given by }%
\begin{eqnarray*}
C(a^{3},a) &=&q(a^{3},a)=1 \\
C(a^{3},a^{2}) &=&q(a^{3},a)^{2}=1 \\
C(a^{3},a^{3}) &=&q(a^{3},a)^{3}=1.
\end{eqnarray*}%
\emph{Similarly, }$T_{a^{3}}(z_{2})=\beta (a^{3},2)z_{2},$\emph{\ with }$%
\beta (a^{3},2)=-1$\emph{; }$T_{a^{3}}(z_{3})=\beta (a^{3},3)z_{3},$\emph{\
with }$\beta (a^{3},3)=-i$\emph{. And }$T_{a^{3}}(z_{4})=\beta
(a^{3},4)z_{3},$\emph{\ with }$\beta (a^{3},4)=1.$

\emph{In the same way, the remaining three cases are analyzed: Case 2, when }%
$q(a,a^{2})=-1$\emph{\ and }$q(a,a^{3})=-1.$ \emph{Case 3, when }$%
q(a,a^{2})=1$\emph{\ and }$q(a,a^{3})=1.$ \emph{Case 4: }$q(a,a^{2})=1$\emph{%
\ and }$q(a,a^{3})=-1.$

\emph{Consequently, there are }$\left\vert A\right\vert ^{n-2}=(2)^{2}$\emph{%
\ graded twisted }$C$\emph{-algebras over }$%
\mathbb{Z}
_{4}.$
\end{example}

Putting all together we obtain

\begin{theorem}
\label{complicado2}Let $W_{1},W_{2}$ be $%
\mathbb{Z}
_{n}$-graded algebras over $k=%
\mathbb{C}
$ or $%
\mathbb{R}
$ with structure constants $C_{1}$ and $C_{2}$ corresponding to the standard
basis $B_{1}$ and $B_{2}$, respectively, and associativity function $r$
which is symmetric in the first two components. Then $W_{1}\cong W_{2}$ as
graded $k$-algebras if and only if $C_{1}(a^{r},a)=C_{2}(a^{r},a),$ for all $%
1\leq r\leq n.$
\end{theorem}

\begin{proof}
By a previous theorem we know that $W_{1}\cong W_{2}$ if and only if $%
d^{2}(C_{1}C_{2}^{-1})=1$; Or equivalently, if and only if $%
d^{2}C_{1}=d^{2}C_{2}$. That is, if and only if, $r_{1}=r_{2}$. Now, if $%
r_{1}=r_{2},$ then equation (\ref{super2}) is true for all $t,$ in
particular if $t=1$. Now, define $f_{i}:G\rightarrow A$ by $%
f_{i}(a^{r})=C_{i}(a^{r},a)$. Applying $d^{1}$ we obtain%
\begin{equation*}
d^{1}f_{i}(a^{r},a^{s})=f_{i}(a^{s})f_{i}(a^{r+s})^{-1}f_{i}(a^{r}).
\end{equation*}%
Hence, $r_{1}=r_{2}$ if and only if $d^{1}f_{1}=d^{1}f_{2};$ Or
equivalently, iff $(f_{1}f_{2}^{-1})\in Ker(d^{1}).$ But 
\begin{eqnarray*}
Ker(d^{1}) &=&H^{1}(G,A)=\{h:G\rightarrow A:h(a)h(b)h(ab)^{-1}=1\} \\
&=&\{h:G\rightarrow A:h\text{ is a group homomorphism}\}.
\end{eqnarray*}%
Therefore, $f_{1}=f_{2}h$ for some group homomorphism $h$. This implies that 
\begin{equation*}
C_{1}(a^{r},a)=h(a)C_{2}(a^{r},a),\text{ for all }1\leq r\leq n-1.
\end{equation*}%
If $r=1,$ we see that $1=C_{1}(a,a)=h(a)C_{2}(a,a),$ and consequently $%
h(a)=1 $. Summarizing, we obtain $C_{1}(a^{r},a)=C_{2}(a^{r},a),$ for $1\leq
r\leq n-1.$
\end{proof}

\subsection{Classification in the real case}

With notation as above: Let $W=\oplus _{g\in G}W_{g}$ be a $G$-graded
twisted $\mathbb{R}$- algebra over a cyclic group $G$ of order $n.$ Let us
fix $a\in G$ a generator for this group, and let $w_{a}$ be a nonzero
element in $W_{a}.$ For $1\leq k\leq n-1,$ we define inductively, $%
w_{a^{k}}=w_{a}\cdot w_{a^{k-1}},$ with $w_{a^{0}}=\alpha $, for some fixed $%
\alpha \neq 0$ in $\mathbb{R}.$ Let us distinguish two cases:

\begin{enumerate}
\item $n$ is odd, or $n$ is even and $\alpha >0$. In this case, we define $%
v_{a^{k}}=\beta ^{k}w_{a^{k}}$ for each $0<k\leq n,$ where $\beta $ is any $%
n $-root of $\alpha ^{-1}$. Clearly, if $k=n,$ then $v_{a^{0}}=\beta
^{n}w_{a^{0}}=\beta ^{n}\alpha =1.$ Moreover, $v_{a}v_{a^{k-1}}=\beta
^{k}w_{a}w_{a^{k-1}}=\beta ^{k}w_{a^{k}}=v_{a^{k}}.$

\item $n$ is even and $\alpha <0.$ Define $v_{a^{0}}=1,$ and for $0<k<n,$
define $v_{a^{k}}=\beta ^{k}w_{a^{k}}$, where $\beta ^{n}=-1/\alpha .$
Hence, for all $k<n$ 
\begin{equation*}
v_{a}v_{a^{k-1}}=\beta ^{k}w_{a}w_{a^{k-1}}=\beta ^{k}w_{a^{k}}=v_{a^{k}}.
\end{equation*}%
And for $k=n$, it holds that $v_{a}v_{a^{n-1}}=\beta
^{n}w_{a}w_{a^{n-1}}=-1/\alpha \cdot \alpha =-1.$
\end{enumerate}

In any of the previous two cases, we will call $B=\{1,v_{a},\ldots
,v_{a^{n-1}}\}$ the \emph{standard basis} for $W.$ Again, for brevity, we
will omit the subindex $B$ in the notation of the structure constant
corresponding to this basis.

Now we notice that in the first case $C(a,a^{r})=1$, for all $r;$ in the
second case, $C(a,a^{r})=1$, if $0\leq r<n-1,$ and $C(a,a^{n-1})=-1.$ This
can be written in a short manner as $C(a^{r},a)=(-1)^{^{\delta (n-1,r)}}$,
where $\delta $ is the delta of Kronecker.

The classification problem in the first case is identical as in the complex
case. Hence, we will restrict to the second case, i.e., when $n$ is even and 
$\alpha <0.$

Since $T_{a}$ sends $v_{a^{k}}$ into $v_{a^{k+1}},$ the linear map $T_{a}$
permutes cyclically the basis $B$, with the exception that it sends $%
v_{a^{n-1}}$ to $-v_{a^{0}}=-1.$ Hence, the minimal polynomial for $T_{a}$
must be $Y^{n}+1=0$. From this, the eigenvalues of $T_{a}$ are the set of
all complex $n$-roots of $-1$, that we will denote by $\{\omega _{1},\ldots
,\omega _{n}\}$. For each $1\leq $ $j\leq n,$ we define $z_{j}=%
\tsum_{k=0}^{n-1}\omega _{j}^{k}v_{a^{k}}.$ Let us see the effect of
multiplying $v_{a^{r}}$ on $z_{j}:$ 
\begin{eqnarray}
v_{a^{r}}\tsum_{k=0}^{n-1}\omega _{j}^{k}v_{a^{k}}
&=&\tsum_{k=0}^{n-1}\omega _{j}^{k}C(a^{r},a^{k})v_{a^{k+r}}  \notag \\
&=&\omega _{j}^{-r}\tsum_{k=0}^{n-r-1}\omega
_{j}^{k+r}C(a^{r},a^{k})v_{a^{k+r}}  \notag \\
&&+\omega _{j}^{-r}\tsum_{k=n-r}^{n-1}\omega
_{j}^{k+r}C(a^{r},a^{k})v_{a^{k+r}}  \label{je1}
\end{eqnarray}%
The first sum of (\ref{je1}) may be rewritten as $\omega
_{j}^{-r}\tsum_{s=r}^{n-1}\omega _{j}^{s}C(a^{r},a^{s-r})v_{a^{s}}$ (with $%
s=k+r$). If $s=k-(n-r)$, we see that the second sum in (\ref{je1}) can be
written as 
\begin{equation*}
(-1)\omega _{j}^{-r}\tsum_{s=0}^{r-1}\omega
_{j}^{s}C(a^{r},a^{s+n-r})v_{a^{n+s}},
\end{equation*}%
which is equal to $(-1)\omega _{j}^{-r}\tsum_{s=0}^{r-1}\omega
_{j}^{s}C(a^{r},a^{n+s-r})v_{a^{s}},$ since $\omega _{j}^{n}=-1.$ Hence, 
\begin{equation}
v_{a^{r}}z_{j}=(-1)\omega _{j}^{-r}\tsum_{s=0}^{r-1}\omega
_{j}^{s}C(a^{r},a^{n+s-r})v_{a^{s}}+\omega _{j}^{-r}\tsum_{s=r}^{n-1}\omega
_{j}^{s}C(a^{r},a^{s-r})v_{a^{s}}  \label{LLL}
\end{equation}

Now, $\omega _{j}^{-1}$ is an eigenvalue of $z_{j}.$ Indeed, if we let $r=1$
in the equation above, 
\begin{eqnarray*}
T_{a}(z_{j}) &=&v_{a}z_{j}=(-1)\omega _{j}^{-1}C(a^{1},a^{n-1})v_{a^{0}} \\
&&+\omega _{j}^{-1}\tsum_{s=1}^{n-1}\omega _{j}^{s}C(a,a^{s-1})v_{a^{s}} \\
&=&\omega _{j}^{-1}v_{a^{0}}+\omega _{j}^{-1}\tsum_{s=1}^{n-1}\omega
_{j}^{s}C(a,a^{s-1})v_{a^{s}}=\omega _{j}^{-1}z_{j}.
\end{eqnarray*}%
since $C(a^{1},a^{s-1})=1,$ when $s=1,\ldots ,n-1.$

On the other hand, form (\ref{maravilla1}) is obtained:

\begin{eqnarray*}
T_{b}(T_{a}(z_{j})) &=&q(b,a)T_{a}(T_{b}(z_{j})) \\
T_{b}(\omega _{j}^{-1}z_{j}) &=&q(b,a)T_{a}(T_{b}(z_{j})) \\
\omega _{j}^{-1}T_{b}(z_{j}) &=&q(b,a)T_{a}(T_{b}(z_{j})) \\
q(a,b)\omega _{j}^{-1}T_{b}(z_{j}) &=&T_{a}(T_{b}(z_{j})).
\end{eqnarray*}%
Hence, $T_{b}(z_{j})$ is an eigenvector associated to the eigenvalue $%
q(a,b)\omega _{j}^{-1}.$But since $T_{a}$ has $n$ different eigenvalues
there must be $\omega _{i}$ and $z_{i}$ such that $\omega
_{i}^{-1}=q(a,b)\omega _{j}^{-1},$ and $T_{b}(z_{j})=\mathcal{\sigma }%
_{b,j}z_{i}$, for some real number $\mathcal{\sigma }_{b,j}\neq 0.$ Letting $%
b=a^{r}$ we get 
\begin{eqnarray}
T_{a^{r}}(z_{j}) &=&\mathcal{\sigma }_{a^{r},j}z_{i}=\mathcal{\sigma }%
_{a^{r},j}\tsum_{k=0}^{n-1}\omega _{i}^{k}v_{a^{k}}  \notag \\
&=&\mathcal{\sigma }_{a^{r},j}\tsum_{s=0}^{n-1}q(a^{r},a)^{s}\omega
_{j}^{s}v_{a^{s}}  \label{LL3}
\end{eqnarray}%
Comparing coefficients in (\ref{LLL}) and (\ref{LL3}) we see that 
\begin{equation}
(-1)\omega _{j}^{-r}C(a^{r},a^{n+s-r})=\mathcal{\sigma }%
_{a^{r},j}q(a^{r},a)^{s},\text{ if }0\leq s\leq r-1,  \label{E1}
\end{equation}%
and 
\begin{equation}
\omega _{j}^{-r}C(a^{r},a^{s-r})=\mathcal{\sigma }_{a^{r},j}q(a^{r},a)^{s},%
\text{ \ if }r\leq s\leq n-1.  \label{E2}
\end{equation}%
Letting $s=r$ in (\ref{E2}), we obtain $\mathcal{\sigma }%
_{a^{r},j}=q(a^{r},a)^{-r}\omega _{j}^{-r}.$ Substituting this values in (%
\ref{E1}) we finally get 
\begin{equation*}
(-1)C(a^{r},a^{n+s-r})=q(a,a^{r})^{r-s}\text{, if }0\leq s\leq r-1.
\end{equation*}%
Letting $k=n+s-r$ we see that 
\begin{equation}
C(a^{r},a^{k})=(-1)q(a,a^{r})^{n-k}=(-1)q(a^{r},a)^{k},\text{ if }n-r\leq
k\leq n-1,  \label{F1}
\end{equation}%
since $q(a^{r},a)^{n}=(\omega _{j}^{-1}\omega _{i})^{n}=1.$ In a similar
manner 
\begin{equation}
C(a^{r},a^{k})=q(a^{r},a)^{k},\text{ if }0\leq k\leq n-r-1.  \label{F2}
\end{equation}

But we know that $C(a,a^{r})=(-1)^{^{\delta (n-1,r)}}$, from which 
\begin{equation*}
q(a^{r},a)=C(a^{r},a)C(a,a^{r})^{-1}=(-1)^{^{\delta (n-1,r)}}C(a^{r},a).
\end{equation*}%
The equalities (\ref{F1} and \ref{F2}) may be written in a simpler form as 
\begin{eqnarray*}
C(a^{r},a^{k}) &=&(-1)^{^{k\delta (n-1,r)}}C(a^{r},a)^{k},\text{ \ \ \ if }%
0\leq k\leq n-r-1. \\
C(a^{r},a^{k}) &=&(-1)^{^{k\delta (n-1,r)+1}}C(a^{r},a)^{k},\text{ \ if }%
n-r\leq k\leq n-1.
\end{eqnarray*}%
From this discussion we deduce the following theorem.

\begin{theorem}
\label{complicado1}Let $G=\left\langle a^{r}:r=0,\ldots ,2m-1\right\rangle $
be a cyclic group of order $n=2m$. Let $W=\oplus _{r=0}^{n-1}W_{a^{r}}$ be a 
$G$-graded twisted real algebra with symmetric associativity function $r$.
Let $B$ denote the standard basis for $W.$ Then $W$ is completely determined
by the function $f:G\rightarrow A,$ $f(a^{r})=C_{B}(a^{r},a).$ Thus, there
are exactly $2\left\vert A\right\vert ^{n-2}$ $G$-graded twisted non
isomorphic algebras over the reals.
\end{theorem}

\begin{acknowledgement}
The authors want to express their gratitud to the Universidad Nacional de
Colombia for their suupport.
\end{acknowledgement}

\end{document}